\documentclass[12pt,leqno,oneside]{amsart}
\usepackage{amsfonts}
\usepackage{amsmath, amssymb}
\usepackage{hyperref}

\numberwithin{equation}{section} 
\newtheorem{theorem}{Theorem}[section]

\newtheorem{lemma}[theorem]{Lemma}
\theoremstyle{definition}

\theoremstyle{remark}
\newtheorem{remark}[theorem]{Remark}
\newtheorem{case}{Case}
\numberwithin{equation}{section}

\input{tcilatex}

\begin{document}
\thanks{The present investigation of the second-named author was supported
by UGC under the grant F.MRP-3977/11 (MRP/UGC-SERO).}
\title[Second Hankel Determinant for bi-Mocanu-convex functions]{Bounds for
the second Hankel determinant of certain bi-univalent functions}
\author{H. Orhan, N. Magesh and J. Yamini}
\address{Department of Mathematics \\
Faculty of Science, Ataturk University \\
25240 Erzurum, Turkey.\\
\texttt{e-mail:} $orhanhalit607@gmail.com$}
\address{Post-Graduate and Research Department of Mathematics,\\
Government Arts College for Men,\\
Krishnagiri 635001, Tamilnadu, India.\\
\texttt{e-mail:} $nmagi\_2000@yahoo.co.in$}
\address{Department of Mathematics,\,\, Govt First Grade College\\
Vijayanagar, Bangalore-560104, Karnataka, India.\\
\texttt{e-mail:} yaminibalaji@gmail.com}

\begin{abstract}
In the present work, we propose to investigate the second Hankel determinant
inequalities for certain class of analytic and bi-univalent functions. Some
interesting applications of the results presented here are also discussed.\ 
\end{abstract}

\maketitle


\section{Introduction}

Let $\mathcal{A}$ denote the class of functions of the form 
\begin{equation}
f(z)=z+\sum\limits_{n=2}^{\infty }a_{n}z^{n}  \label{Int-e1}
\end{equation}%
which are analytic in the open unit disc $\mathbb{U}=\{z:z\in \mathbb{C}\,\,%
\mathrm{and}\,\,|z|<1\}.$ Further, by $\mathcal{S}$ we will show the family
of all functions in $\mathcal{A}$ which are univalent in $\mathbb{U}.$


Some of the important and well-investigated subclasses of the univalent
function class $\mathcal{S}$ include (for example) the class $\mathcal{S}%
^{\ast }(\beta )$ of starlike functions of order $\beta $ in $\mathbb{U}$
and the class $\mathcal{K}(\beta )$ of convex functions of order $\beta $ in 
$\mathbb{U}.$ By definition, we have 
\begin{equation*}
\mathcal{S}^{\ast }(\beta ):=\left\{ f:f\in \mathcal{A}\,\text{\ }\mathrm{and%
}\,\Re \left( \frac{zf^{\prime }(z)}{f(z)}\right) >\beta ;\,\,z\in \mathbb{U}%
;\,\,\,0\leq \beta <1\right\}
\end{equation*}%
and 
\begin{equation*}
\mathcal{K}(\beta ):=\left\{ f:f\in \mathcal{A}\,\text{\ }\mathrm{and}\,\Re
\left( 1+\frac{zf^{\prime \prime }(z)}{f^{\prime }(z)}\right) >\beta
;\,\,z\in \mathbb{U};\,\,\,0\leq \beta <1\right\} .  \label{CV-e}
\end{equation*}

The arithmetic means of some functions and expressions is very frequently
used in mathematics, specially in geometric function theory. Making use of
the arithmetic means Mocanu \cite{Mocanu} introduced the class of $\alpha-$%
convex ($0\leqq \alpha \leqq 1$) functions (later called as Mocanu-convex
functions) as follows: 
\begin{equation*}
\mathcal{M}(\alpha):= \left \{ f : f \in \mathcal{S} \, \mathrm{and}\, \Re
\left ((1-\alpha)\frac{zf^{\prime }(z)}{f(z)} +\alpha \left(1+\frac{%
zf^{\prime \prime }(z)}{f^{\prime }(z)}\right)\right ) > 0 ; \,\, z \in 
\mathbb{U}\right \}.
\end{equation*}

In \cite{SSM-PTM-MOR}, it was shown that if the above analytical criteria
holds for $z\in \mathbb{U},$ then $f$ is in the class of starlike functions $%
\mathcal{S}^*(0)$ for $\alpha$ real and is in the class of convex functions $%
\mathcal{K}(0)$ for $\alpha \geq 1.$ In general, the class of $\alpha-$%
convex functions determines the arithmetic bridge between starlikeness and
convexity.


It is well known that every function $f\in \mathcal{S}$ has an inverse $%
f^{-1},$ defined by 
\begin{equation*}
f^{-1}(f(z))=z \qquad \left(z \in \mathbb{U}\right)
\end{equation*}
and 
\begin{equation*}
f(f^{-1}(w))=w \qquad \left(|w| < r_0(f);\,\, r_0(f) \geq \frac{1}{4}\right),
\end{equation*}
where 
\begin{equation}  \label{Int-f-inver}
f^{-1}(w) = w - a_2w^2 + (2a_2^2-a_3)w^3 - (5a_2^3-5a_2a_3+a_4)w^4+\ldots .
\end{equation}

A function $f \in \mathcal{A}$ is said to be bi-univalent in $\mathbb{U}$ if
both $f(z)$ and $f^{-1}(z)$ are univalent in $\mathbb{U}.$ Let $\sigma$
denote the class of bi-univalent functions in $\mathbb{U}$ given by (\ref%
{Int-e1}).

For $0\leq \beta <1,$ a function $f\in\sigma$ is in the class $%
S^*_{\sigma}(\beta)$ of bi-starlike function of order $\beta,$ or $\mathcal{K%
}_{\sigma, \beta}$ of bi-convex function of order $\beta$ if both $f$ and $%
f^{-1}$ are respectively starlike or convex functions of order $\beta.$
Also, a function $f$ is in the class $\mathcal{M}^{\alpha}_{\Sigma}(\beta)$
of bi-Mocanu convex function of order $\beta$ if both $f$ and $f^{-1}$ are
respectively Mocanu convex function of order $\beta.$ For a brief history
and interesting examples of functions which are in (or which are not in) the
class $\sigma $, together with various other properties of the bi-univalent
function class $\sigma $ one can refer the work of Srivastava et al. \cite%
{HMS-AKM-PG} and references therein. Various subclasses of the bi-univalent
function class $\sigma $ were introduced and non-sharp estimates on the
first two coefficients $|a_{2}|$ and $|a_{3}|$ in the Taylor-Maclaurin
series expansion (\ref{Int-e1}) were found in several recent investigations
(see, for example, \cite{Ali-Ravi-Ma-Mina-class,Caglar-Orhan,Deniz-SHD,
BAF-MKA,Li-Wang,HO-NM-VKB,Peng}). However, the problem to find the
coefficient bounds on $|a_{n}|$ ($n=3,4,\dots $) for functions $f\in \sigma $
is still an open problem.


For integers $n \geq 1$ and $q \geq 1,$ the $q-$th Hankel determinant,
defined as 
\begin{equation*}  \label{HD}
H_{q}(n)=\left|%
\begin{array}{cccc}
a_{n} & a_{n+1} & \cdots & a_{n+q-1} \\ 
a_{n+1} & a_{n+2} & \cdots & a_{n+q-2} \\ 
\vdots & \vdots & \vdots & \vdots \\ 
a_{n+q-1} & a_{n+q-2} & \cdots & a_{n+2q-2}%
\end{array}
\right| \qquad (a_1=1).
\end{equation*}
The Hankel determinant plays an important role in the study of singularities
(see \cite{Dienes-1957}). This is also an important in the study of power
series with integral coefficients \cite{Cantor,Dienes-1957}. The properties
of the Hankel determinants can be found in \cite{Vein-Dale}. 
The Hankel determinants $H_2(1) = a_3 - a_2^2$ and $H_2(2) = a_2a_4 - a_2^3$
are well-known as Fekete-Szeg\"{o} and second Hankel determinant functionals
respectively. Further Fekete and Szeg\"{o} \cite{Fekete-Szego} introduced
the generalized functional $a_3-\delta a_2^2,$ where $\delta$ is some real
number. In 1969, Keogh and Merkes \cite{Keogh-Merkes} discussed the
Fekete-Szeg\"{o} problem for the classes $\mathcal{S}^*$ and $\mathcal{K}.$
Recently, several authors have investigated upper bounds for the Hankel
determinant of functions belonging to various subclasses of univalent
functions \cite{SHD-Ali-2009,VKD-RRT-SHD-2014,
Janteng-2007,SHD-Lee-2013,GMS-NM-SHD,Orhan-FHD-2010} and the references
therein. On the other hand, Zaprawa \cite{Zaprawa,Zaprawa-AAA} extended the
study on Fekete-Szeg\"{o} problem for certain subclasses of bi-univalent
function class $\sigma .$ Following Zaprawa \cite{Zaprawa,Zaprawa-AAA}, the
Fekete-Szeg\"{o} problem for functions belonging to various other subclasses
of bi-univalent functions were considered in \cite%
{Altinkaya,Jay-NM-JY,HO-NM-VKB-Fekete}. Very recently, the upper bounds of $%
H_2(2)$ for the classes $S^*_{\sigma}(\beta)$ and $K_{\sigma}(\beta)$ were
discussed by Deniz et al. \cite{Deniz-SHD}. 

Next we state the following lemmas we shall use to establish the desired
bounds in our study.

\begin{lemma}
\cite{Pom}\label{L-Repart-fun-p} If the function $p\in \mathcal{P}$ is given
by the series 
\begin{equation}  \label{Repart-fun-p}
p(z)= 1+ c_{1} z + c_{2} z^{2} + c_{3} z^{3} + \cdots,
\end{equation}
then the following sharp estimate holds: 
\begin{equation}  \label{cnbound}
|c_{k}|\leq 2, \qquad k= 1,2, \cdots.
\end{equation}
\end{lemma}

\begin{lemma}
\label{L-C2-C3} \cite{Grender} If the function $p\in \mathcal{P}$ is given
by the series (\ref{Repart-fun-p}), then 
\begin{eqnarray*}
2c_2 &=& c_1^2 + x (4-c_1^2)\,\,  \label{L2-c2} \\
4c_3 &=& c_1^3+2c_1(4-c_1^2)x-c_1(4-c_1^2)x^2+2(4-c_1^2)(1-|x|^2)z\,\,
\label{L2-c3}
\end{eqnarray*}
for some $x,$ $z$ with $|x| \leq 1$ and $|z| \leq 1.$
\end{lemma}

Inspired by the works of \cite{Deniz-SHD,Zaprawa} we consider the following
subclass of the function class $\sigma.$ 

For $0 \leq \alpha \leq 1$ and $0\leq \beta <1,$ a function $f\in \sigma$
given by (\ref{Int-e1}) is said to be in the class $\mathcal{M}%
^{\alpha}_{\sigma}(\beta)$ if the following conditions are satisfied: 
\begin{equation*}
\Re \left ( (1-\alpha)\frac{zf^{\prime }(z)}{f(z)}+\alpha\left(1+\frac{%
zf^{\prime \prime }(z)}{f^{\prime }(z)}\right ) \right)\geq \beta \qquad (z
\in \mathbb{U})
\end{equation*}
and for $g=f^{-1}$ 
\begin{equation*}
\Re \left ( (1-\alpha)\frac{wg^{\prime }(w)}{g(w)}+\alpha\left(1+\frac{%
wg^{\prime \prime }(w)}{g^{\prime }(w)}\right ) \right) \geq \beta \qquad (w
\in \mathbb{U}).
\end{equation*}
The class was introduced and studied by Li and Wang \cite{Li-Wang}, further
the study was extended by Ali et al. \cite{Ali-Ravi-Ma-Mina-class}. In this
paper we shall obtain the functional $H_2(2)$ for functions $f$ belongs to
the class $\mathcal{M}^{\alpha}_{\sigma}(\beta)$ and its special classes.

\section{Bounds for the second Hankel determinant}

We begin this section with the following theorem: 

\begin{theorem}
\label{th-SHD-class} Let $f$ of the form (\ref{Int-e1}) be in $\mathcal{M}%
^{\alpha}_{\sigma}(\beta).$ Then 
\begin{equation*}
|a_2a_4-a_3^2| \leq \left \{ 
\begin{array}{llll}
\frac{4(1-\beta)^2}{3(1+\alpha)^3(1+3\alpha)} \left[4(1-\beta)^2+(1+\alpha)^2%
\right]~; &  &  &  \\ 
\qquad \qquad\qquad\beta \in \left[0, 1- \frac{(1+\alpha)[3(1+3\alpha)+\sqrt{%
9(1+3\alpha)^2-48(1+\alpha) (1+3\alpha)+128(1+2\alpha)^2}]}{16(1+2\alpha)}%
\right] &  &  &  \\ 
\tfrac{(1-\beta)^2}{(1+\alpha)(1+3\alpha)}\tfrac{[(1-\beta)^2(1+3%
\alpha)(13+7\alpha)-12(1-\beta)(1+\alpha)(1+2\alpha)(1+3\alpha)-4(1+%
\alpha)^2(9\alpha^2+8\alpha+2)]} {[16(1-\beta)^2(1+2\alpha)-6(1-\beta)(1+%
\alpha)(1+3\alpha)](1+2\alpha)
+(1+\alpha)^2[3(1+\alpha)(1+3\alpha)-8(1+2\alpha)^2]}~; &  &  &  \\ 
\qquad \qquad\qquad \beta \in \left(1- \frac{(1+\alpha)[3(1+3\alpha)+\sqrt{%
9(1+3\alpha)^2+128(1+2\alpha)^2}]}{32(1+2\alpha)}, 1\right). &  &  & 
\end{array}%
\right.
\end{equation*}
\end{theorem}

\begin{proof}
Let $f\in \mathcal{M}^{\alpha}_{\sigma}(\beta).$ Then 
\begin{equation}  \label{SHD-th-p-e1}
(1-\alpha)\frac{zf^{\prime }(z)}{f(z)}+\alpha\left(1+\frac{zf^{\prime \prime
}(z)}{f^{\prime }(z)}\right )=\beta+(1-\beta)p(z)
\end{equation}
and 
\begin{equation}  \label{SHD-th-p-e2}
(1-\alpha)\frac{wg^{\prime }(w)}{g(w)}+\alpha\left(1+\frac{wg^{\prime \prime
}(w)}{g^{\prime }(w)}\right )=\beta+(1-\beta)q(w),
\end{equation}
where $p, q \in \mathcal{P}$ and defined by 
\begin{equation}  \label{SHD-th-p-e3}
p(z)=1+c_1z+c_2z^2+c_3z^3+\dots
\end{equation}
and 
\begin{equation}  \label{SHD-th-p-e4}
q(z)=1+d_1w+d_2w^2+d_3w^3+\dots
\end{equation}
It follows from (\ref{SHD-th-p-e1}), (\ref{SHD-th-p-e2}), (\ref{SHD-th-p-e3}%
) and (\ref{SHD-th-p-e4}) that 
\begin{eqnarray}
(1+\alpha)a_2 &=& (1-\beta)c_1  \label{SHD-th-p-e5} \\
2(1+2\alpha)a_3-(1+3\alpha)a_2^2 &=& (1-\beta)c_2  \label{SHD-th-p-e6} \\
3(1+3\alpha)a_4-3(1+5\alpha)a_2a_3+(1+7\alpha)a_2^3 &=& (1-\beta)c_3
\label{SHD-th-p-e7}
\end{eqnarray}
and 
\begin{eqnarray}
-(1+\alpha)a_2 &=& (1-\beta)d_1  \label{SHD-th-p-e8} \\
(3+5\alpha)a_2^2-(2+4\alpha)a_3 &=& (1-\beta)d_2  \label{SHD-th-p-e9} \\
(12+30\alpha)a_2a_3-(10+22\alpha)a_2^3-(3+9\alpha)a_4 &=& (1-\beta)d_3.
\label{SHD-th-p-e10}
\end{eqnarray}
From (\ref{SHD-th-p-e5}) and (\ref{SHD-th-p-e8}), we find that 
\begin{equation}  \label{SHD-th-p-e11}
c_1 =-d_1
\end{equation}
and 
\begin{equation}  \label{SHD-a2}
a_2=\frac{1-\beta}{1+\alpha}c_1.
\end{equation}
Now, from (\ref{SHD-th-p-e6}), (\ref{SHD-th-p-e9}) and (\ref{SHD-a2}), we
have 
\begin{equation}  \label{SHD-a3}
a_3=\frac{(1-\beta)^2}{(1+\alpha)^2}c_1^2+ \frac{1-\beta}{4+8\alpha}%
(c_2-d_2).
\end{equation}
Also, from (\ref{SHD-th-p-e7}) and (\ref{SHD-th-p-e10}), we find that 
\begin{equation}  \label{SHD-a4}
a_4=\frac{(2+8\alpha)(1-\beta)^3}{(3+9\alpha)(1+\alpha)^3}c_1^3 +\frac{%
5(1-\beta)^2}{8(1+\alpha)(1+2\alpha)}c_1(c_2-d_2) +\frac{1-\beta}{%
6(1+3\alpha)}(c_3-d_3).
\end{equation}
Then, we can establish that 
\begin{eqnarray}  \label{H2of2}
|a_2a_4-a_3^2| =&& \left | \frac{-1}{3}\frac{(1-\beta)^4}{%
(1+\alpha)^3(1+3\alpha)}c_1^4 +\frac{(1-\beta)^3}{8(1+\alpha)^2(1+2\alpha)}%
c_1^2(c_2-d_2)\right.  \notag \\
&& \left. \quad +\frac{(1-\beta)^2}{6(1+\alpha)(1+3\alpha)}c_1(c_3-d_3) -%
\frac{(1-\beta)^2}{16(1+2\alpha)^2}(c_2-d_2)^2 \right |.
\end{eqnarray}

According to Lemma \ref{L-C2-C3} and (\ref{SHD-th-p-e11}), we write

\begin{equation}  \label{c2=d2}
c_2-d_2=\frac{(4-c_1^2)}{2}(x-y)
\end{equation}
and 
\begin{equation}  \label{c3-d3}
c_3-d_3=\frac{c_1^3}{2}+\frac{c_1(4-c_1^2)(x+y)}{2}-\frac{%
c_1(4-c_1^2)(x^2+y^2)}{4}+\frac{(4-c_1^2)[(1-|x|^2)z-(1-|y|^2)w]}{2}
\end{equation}
for some $x,y,z$ and $w$ with $|x|\leq 1,$ $|y|\leq 1,$ $|z|\leq 1$ and $%
|w|\leq 1.$ Using (\ref{c2=d2}) and (\ref{c3-d3}) in \ref{H2of2}, we have

\begin{eqnarray}  \label{H2-2}
|a_2a_4-a_3^2| &=& \left | \frac{-(1-\beta)^4c_1^4}{3(1+\alpha)^3(1+3\alpha)}
+\frac{(1-\beta)^3c_1^2(4-c_1^2)(x-y)}{16(1+\alpha)^2(1+2\alpha)} +\frac{%
(1-\beta)^2c_1}{6(1+\alpha)(1+3\alpha)} \right.  \notag \\
&& \left. \quad \times \left[\frac{c_1^3}{2}+\frac{c_1(4-c_1^2)(x+y)}{2}-%
\frac{c_1(4-c_1^2)(x^2+y^2)}{4}\right. \right.  \notag \\
&& \left. \left. \qquad +\frac{(4-c_1^2)[(1-|x|^2)z-(1-|y|^2)w]}{2}\right]-%
\frac{(1-\beta)^2(4-c_1^2)^2}{64(1+2\alpha)^2}(x-y)^2 \right |  \notag \\
&\leq& \frac{(1-\beta)^4}{3(1+\alpha)^3(1+3\alpha)}c_1^4 +\frac{%
(1-\beta)^2c_1^4}{12(1+\alpha)(1+3\alpha)} + \frac{(1-\beta)^2c_1(4-c_1^2)}{%
6(1+\alpha)(1+3\alpha)}  \notag \\
&& +\left[\frac{(1-\beta)^3c_1^2(4-c_1^2)}{16(1+\alpha)^2(1+2\alpha)} +\frac{%
(1-\beta)^2c_1^2(4-c_1^2)}{12(1+\alpha)(1+3\alpha)}\right](|x|+|y|)  \notag
\\
&& +\left[\frac{(1-\beta)^2c_1^2(4-c_1^2)}{24(1+\alpha)(1+3\alpha)} -\frac{%
(1-\beta)^2c_1(4-c_1^2)}{12(1+\alpha)(1+3\alpha)}\right](|x|^2+|y|^2)  \notag
\\
&& \quad +\frac{(1-\beta)^2(4-c_1^2)^2}{64(1+2\alpha)^2}(|x|+|y|)^2.  \notag
\end{eqnarray}

Since $p\in \mathcal{P},$ so $|c_{1}|\leq 2.$ Letting $c_{1}=c,$ we may
assume without restriction that $c\in \lbrack 0,2].$ Thus, for $\gamma
_{1}=|x|\leq 1$ and $\gamma _{2}=|y|\leq 1,$ we obtain 
\begin{equation*}
|a_{2}a_{4}-a_{3}^{2}|\leq T_{1}+T_{2}(\gamma _{1}+\gamma _{2})+T_{3}(\gamma
_{1}^{2}+\gamma _{2}^{2})+T_{4}(\gamma _{1}+\gamma _{2})^{2}=F(\gamma
_{1},\gamma _{2}),
\end{equation*}%
\begin{eqnarray*}
T_{1} &=&T_{1}(c)=\frac{(1-\beta )^{4}}{3(1+\alpha )^{3}(1+3\alpha )}c^{4}+%
\frac{(1-\beta )^{2}c^{4}}{12(1+\alpha )(1+3\alpha )}+\frac{(1-\beta
)^{2}c(4-c^{2})}{6(1+\alpha )(1+3\alpha )}\geq 0 \\
T_{2} &=&T_{2}(c)=\frac{(1-\beta )^{3}c^{2}(4-c^{2})}{16(1+\alpha
)^{2}(1+2\alpha )}+\frac{(1-\beta )^{2}c^{2}(4-c^{2})}{12(1+\alpha
)(1+3\alpha )}\geq 0 \\
T_{3} &=&T_{3}(c)=\frac{(1-\beta )^{2}c^{2}(4-c^{2})}{24(1+\alpha
)(1+3\alpha )}-\frac{(1-\beta )^{2}c(4-c^{2})}{12(1+\alpha )(1+3\alpha )}%
\leq 0 \\
T_{4} &=&T_{4}(c)=\frac{(1-\beta )^{2}(4-c^{2})^{2}}{64(1+2\alpha )^{2}}\geq
0.
\end{eqnarray*}%
Now we need to maximize $F(\gamma _{1},\gamma _{2})$ in the closed square $%
\mathbb{S}:=\{(\gamma _{1},\gamma _{2}):0\leq \gamma _{1}\leq 1,0\leq \gamma
_{2}\leq 1\}$ for $c\in \lbrack 0,2]$. We must investigate the maximum of $%
F(\gamma _{1},\gamma _{2})$ according to $c\in (0,2),$ $c=0$ and $c=2$
taking into account the sign of $F_{\gamma _{1}\gamma _{1}}F_{\gamma
_{2}\gamma _{2}}-(F_{\gamma _{1}\gamma _{2}})^{2}.$

Firstly, let $c\in (0,2).$ Since $T_{3}<0$ and $T_{3}+2T_{4}>0$ for $c\in
(0,2),$ we conclude that 
\begin{equation*}
F_{\gamma _{1}\gamma _{1}}F_{\gamma _{2}\gamma _{2}}-(F_{\gamma _{1}\gamma
_{2}})^{2}<0.
\end{equation*}

Thus, the function $F$ cannot have a local maximum in the interior of the
square $\mathbb{S}.$ Now, we investigate the maximum of $F$ on the boundary
of the square $\mathbb{S}.$

For $\gamma_1=0$ and $0\leq \gamma_2 \leq 1$ (similarly $\gamma_2=0$ and $%
0\leq \gamma_1\leq1$) we obtain 
\begin{equation*}
F(0,\gamma_2)=G(\gamma_2)=T_1+T_2\gamma_2+(T_3+T_4)\gamma_2^2
\end{equation*}

(i) The case $T_{3}+T_{4}\geq 0:$ In this case for $0<\gamma _{2}<1$ and any
fixed $c$ with $0<c<2,$ it is clear that $G^{\prime }(\gamma
_{2})=2(T_{3}+T_{4})\gamma _{2}+T_{2}>0,$ that is, $G(\gamma _{2})$ is an
increasing function. Hence, for fixed $c\in (0,2),$ the maximum of $G(\gamma
_{2})$ occurs at $\gamma _{2}=1$ and 
\begin{equation*}
\max G(\gamma _{2})=G(1)=T_{1}+T_{2}+T_{3}+T_{4}.
\end{equation*}

(ii) The case $T_{3}+T_{4}<0:$ Since $T_{2}+2(T_{3}+T_{4})\geq 0$ for $%
0<\gamma _{2}<1$ and any fixed $c$ with $0<c<2,$ it is clear that $%
T_{2}+2(T_{3}+T_{4})<2(T_{3}+T_{4})\gamma _{2}+T_{2}<T_{2}$ and so $%
G^{\prime }(\gamma _{2})>0.$ Hence for fixed $c\in (0,2),$ the maximum of $%
G(\gamma _{2})$ occurs at $\gamma _{2}=1$ and

Also for $c=2$ we obtain 
\begin{equation}  \label{F}
F(\gamma_1, \gamma_2) = \frac{4(1-\beta)^2}{3(1+\alpha)^3(1+3\alpha)} \left[%
4(1-\beta)^2+(1+\alpha)^2\right].
\end{equation}

Taking into account the value (\ref{F}) and the cases $i$ and $ii,$ for $%
0\leq \gamma_2 <1$ and any fixed $c$ with $0\leq c \leq 2,$ 
\begin{equation*}
\max G(\gamma_2) = G(1) = T_1+T_2+T_3+T_4.
\end{equation*}

For $\gamma_1=1$ and $0\leq \gamma_2 \leq 1$ (similarly $\gamma_2=1$ and $%
0\leq \gamma_1 \leq 1$), we obtain 
\begin{equation*}
F(1,\gamma_2)=H(\gamma_2)=(T_3+T_4)\gamma_2^2 + (T_2+2T_4)\gamma_2 + T_1 +
T_2 + T_3 + T_4.
\end{equation*}

Similarly, to the above cases of $T_3+T_4,$ we get that 
\begin{equation*}
\max H(\gamma_2) = H(1) = T_1+2T_2+2T_3+4T_4.
\end{equation*}

Since $G(1)\leq H(1)$ for $c\in (0,2),$ $\max F(\gamma _{1},\gamma
_{2})=F(1,1)$ on the boundary of the square $\mathbb{S}.$ Thus the maximum
of $F$ occurs at $\gamma _{1}=1$ and $\gamma _{2}=1$ in the closed square $%
\mathbb{S}.$

Let $K:(0,2)\rightarrow \mathbb{R}$ 
\begin{equation}
K(c)=\max F(\gamma _{1},\gamma _{2})=F(1,1)=T_{1}+2T_{2}+2T_{3}+4T_{4}.
\label{K}
\end{equation}

Substituting the values of $T_1,$ $T_2,$ $T_3$ and $T_4$ in the function $K$
defined by (\ref{K}), yields 
\begin{align*}
K(c) &= \frac{(1-\beta)^2}{48(1+\alpha)^3(1+2\alpha)^2(1+3\alpha)}\left\{%
\left[ 16(1-\beta)^2(1+2\alpha)^2 \right.\right. \\
& \left.\left.\qquad -6(1-\beta)(1+\alpha)(1+2\alpha)(1+3\alpha)
-8(1+\alpha)^2(1+2\alpha)^2+3(1+\alpha)^3(1+3\alpha)\right] c^4 \right. \\
& \left. \qquad + 24(1+\alpha)\left[(1-\beta)(1+2\alpha)(1+3\alpha)
+2(1+\alpha)(1+2\alpha)^2-(1+\alpha)^2(1+3\alpha)\right] c^2 \right. \\
& \qquad \left.+ 48(1+\alpha)^3(1+3\alpha)\right\}.
\end{align*}

Assume that $K(c)$ has a maximum value in an interior of $c\in (0,2),$ by
elementary calculation, we find 
\begin{align*}
K^{\prime }(c)& =\frac{(1-\beta )^{2}}{12(1+\alpha )^{3}(1+2\alpha
)^{2}(1+3\alpha )}\left\{ \left[ 16(1-\beta )^{2}(1+2\alpha )^{2}\right.
\right. \\
& \left. \qquad -6(1-\beta )(1+\alpha )(1+2\alpha )(1+3\alpha )-8(1+\alpha
)^{2}(1+2\alpha )^{2}+3(1+\alpha )^{3}(1+3\alpha )\right] c^{3} \\
& \qquad +12(1+\alpha )\left. \left[ (1-\beta )(1+2\alpha )(1+3\alpha
)+2(1+\alpha )(1+2\alpha )^{2}-(1+\alpha )^{2}(1+3\alpha )\right] c\right\} .
\end{align*}

After some calculations we concluded following cases:

\begin{case}
Let

\begin{equation*}
[16(1-\beta)^2(1+2\alpha)-6(1-\beta)(1+\alpha)(1+3\alpha)](1+2\alpha)
+(1+\alpha)^2[3(1+\alpha)(1+3\alpha)-8(1+2\alpha)^2]\geq 0,
\end{equation*}

\noindent that is,

\begin{equation*}
\beta \in \left[ 0,1-\frac{(1+\alpha )[3(1+3\alpha )+\sqrt{9(1+3\alpha
)^{2}-48(1+\alpha )(1+3\alpha )+128(1+2\alpha )^{2}}]}{16(1+2\alpha )}\right]
.
\end{equation*}%
Therefore $K^{\prime }(c)>0$ for $c\in (0,2).$ Since $K$ is an increasing
function in the interval $(0,2),$ maximum point of $K$ must be on the
boundary of $c\in (0,2),$ that is, $c=2.$ Thus, we have 
\begin{equation*}
\max\limits_{0<c<2}K(c)=K(2)=\frac{4(1-\beta )^{2}}{3(1+\alpha
)^{3}(1+3\alpha )}\left[ 4(1-\beta )^{2}+(1+\alpha )^{2}\right] .
\end{equation*}
\end{case}

\begin{case}
Let

\begin{equation*}
[16(1-\beta)^2(1+2\alpha)-6(1-\beta)(1+\alpha)(1+3\alpha)](1+2\alpha)
+(1+\alpha)^2[3(1+\alpha)(1+3\alpha)-8(1+2\alpha)^2] < 0,
\end{equation*}

\noindent that is,

\begin{equation*}
\beta \in \left[1- \frac{(1+\alpha)[3(1+3\alpha)+\sqrt{9(1+3\alpha)^2-48(1+%
\alpha)(1+3\alpha)+128(1+2\alpha)^2}]}{16(1+2\alpha)}, 1\right].
\end{equation*}

Then $K^{\prime }(c)=0$ implies the real critical point $c_{0_1}=0$ or 
\begin{equation*}
c_{0_2} = \sqrt{\tfrac{-12(1+\alpha)[(1-\beta)(1+2\alpha)(1+3\alpha)+2(1+%
\alpha)(1+2\alpha)^2-(1+\alpha)^2(1+3\alpha)]} {[16(1-\beta)^2(1+2%
\alpha)-6(1-\beta)(1+\alpha)(1+3\alpha)](1+2\alpha)
+(1+\alpha)^2[3(1+\alpha)(1+3\alpha)-8(1+2\alpha)^2]}} .
\end{equation*}
When 
\begin{equation*}
\beta \in \left(1- \tfrac{(1+\alpha)[3(1+3\alpha)+\sqrt{9(1+3\alpha)^2-48(1+%
\alpha)(1+3\alpha)+128(1+2\alpha)^2}]}{16(1+2\alpha)} \right., \left.1- 
\tfrac{(1+\alpha)[3(1+3\alpha)+\sqrt{9(1+3\alpha)^2+128(1+2\alpha)^2}]}{%
32(1+2\alpha)}\right],
\end{equation*}

we observe that $c_{0_2}\geq 2,$ that is, $c_{0_2}$ is out of the interval $%
(0,2).$ Therefore, the maximum value of $K(c)$ occurs at $c_{0_1}=0$ or $%
c=c_{0_2}$ which contradicts our assumption of having the maximum value at
the interior point of $c \in [0,2].$

When $\beta \in \left(1- \frac{(1+\alpha)[3(1+3\alpha)+\sqrt{%
9(1+3\alpha)^2+128(1+2\alpha)^2}]}{32(1+2\alpha)}, 1\right),$ we observe
that $c_{0_2}< 2,$ that is, $c_{0_2}$ is an interior of the interval $[0,2].$
Since $K^{\prime \prime }(c_{0_2})<0,$ the maximum value of $K(c)$ occurs at 
$c=c_{0_2}.$ Thus, we have 
\begin{eqnarray*}
\max\limits_{0 \leq c \leq 2} K(c) &=& K(c_{0_2}) \\
&=& \tfrac{(1-\beta)^2}{(1+\alpha)(1+3\alpha)}\tfrac{[(1-\beta)^2(1+3%
\alpha)(13+7\alpha)-12(1-\beta)(1+\alpha)(1+2\alpha)(1+3\alpha)-4(1+%
\alpha)^2(9\alpha^2+8\alpha+2)]} {[16(1-\beta)^2(1+2\alpha)-6(1-\beta)(1+%
\alpha)(1+3\alpha)](1+2\alpha)
+(1+\alpha)^2[3(1+\alpha)(1+3\alpha)-8(1+2\alpha)^2]}.
\end{eqnarray*}
\end{case}

This completes the proof.
\end{proof}

\begin{remark}
For $\alpha=0$ and $\alpha=1,$ Theorem \ref{th-SHD-class} would reduce to a
known results in \cite[Theorem 2.1, Theorem 2.3]{Deniz-SHD}.
\end{remark}



\begin{thebibliography}{99}
\bibitem{SHD-Ali-2009} R. M. Ali, S. K. Lee, V. Ravichandran and S.
Supramaniam, The Fekete-Szeg\H o coefficient functional for transforms of
analytic functions, Bull. Iranian Math. Soc. \textbf{35} (2009), no.~2,
119--142, 276.

\bibitem{Ali-Ravi-Ma-Mina-class} R. M. Ali, S. K. Lee, V. Ravichandran and
S. Supramanian, Coefficient estimates for bi-univalent Ma-Minda starlike and
convex functions, Appl. Math. Lett. \textbf{25} (2012), no.~3, 344--351.

\bibitem{Altinkaya} A. Altinkaya and S. Yal\c{c}in, Fekete-Szeg\"{o}
inequalities for classes of bi-univalent functions defined by subordination,
Adv. Math. Sci. J. \textbf{3} (2014), no.2, 63--71

\bibitem{Cantor} D. G. Cantor, Power series with integral coefficients,
Bull. Amer. Math. Soc. \textbf{69} (1963), 362--366.

\bibitem{Caglar-Orhan} M. \c Ca\u glar, H. Orhan\ and\ N. Ya\u gmur,
Coefficient bounds for new subclasses of bi-univalent functions, Filomat 
\textbf{27} (2013), no.~7, 1165--1171.

\bibitem{VKD-RRT-SHD-2014} V. K. Deekonda and R. Thoutreedy, An upper bound
to the second Hankel determinant for functions in Mocanu class, Vietnam J.
Math., April-2014 (On line version).

\bibitem{Deniz-SHD} E. Deniz, M. \c{C}a\u{g}lar and H. Orhan, Second Hankel
determinant for bi-starlike and bi-convex functions of order $\beta,$
arXiv:1501.01682v1.

\bibitem{Dienes-1957} P. Dienes, \textit{The Taylor series: an introduction
to the theory of functions of a complex variable}, Dover, New York, 1957.

\bibitem{Fekete-Szego} M. Fekete\ and\ G. Szeg\"o, Eine Bemerkung Uber
Ungerade Schlichte Funktionen, J. London Math. Soc. \textbf{S1-8} (1933),
no.~2, 85--89.

\bibitem{BAF-MKA} B. A. Frasin\ and\ M. K. Aouf, New subclasses of
bi-univalent functions, Appl. Math. Lett. \textbf{24} (2011), no.~9,
1569--1573.

\bibitem{Grender} U. Grenander\ and\ G. Szeg\"o, \textit{Toeplitz forms and
their applications}, California Monographs in Mathematical Sciences, Univ.
California Press, Berkeley, 1958.

\bibitem{Jay-NM-JY} J. M. Jahangiri, N. Magesh\ and\ J. Yamini,
Fekete-Szeg\"o inequalities for classes of bi-starlike and bi-convex
functions, Electron. J. Math. Anal. Appl. \textbf{3} (2015), no.~1, 133--140.

\bibitem{Janteng-2007} A. Janteng, S. A. Halim\ and\ M. Darus, Hankel
determinant for starlike and convex functions, Int. J. Math. Anal. (Ruse) 
\textbf{1} (2007), no.~13-16, 619--625.

\bibitem{Keogh-Merkes} F. R. Keogh\ and\ E. P. Merkes, A coefficient
inequality for certain classes of analytic functions, Proc. Amer. Math. Soc. 
\textbf{20} (1969), 8--12.

\bibitem{SHD-Lee-2013} K. Lee, V. Ravichandran and S. Supramaniam, Bounds
for the second Hankel determinant of certain univalent functions, J.
Inequal. Appl. \textbf{281} (2013), 1--17.

\bibitem{Li-Wang} X.-F. Li\ and\ A.-P. Wang, Two new subclasses of
bi-univalent functions, Int. Math. Forum \textbf{7} (2012), no.~29-32,
1495--1504.

\bibitem{GMS-NM-SHD} G. Murugusundaramoorthy\ and\ N. Magesh, Coefficient
inequalities for certain classes of analytic functions associated with
Hankel determinant, Bull. Math. Anal. Appl. \textbf{1} (2009), no.~3, 85--89.

\bibitem{Orhan-FHD-2010} H. Orhan, E. Deniz\ and\ D. Raducanu, The
Fekete-Szeg\"o problem for subclasses of analytic functions defined by a
differential operator related to conic domains, Comput. Math. Appl. \textbf{%
59} (2010), no.~1, 283--295.

\bibitem{HO-NM-VKB} H. Orhan, N. Magesh and V. K. Balaji, Initial
coefficient bounds for a general class of bi-univalent functions,
arXiv:1303.2527v1.

\bibitem{HO-NM-VKB-Fekete} H. Orhan, N. Magesh and V. K. Balaji,
Fekete-Szeg\H o problem for certain classes of Ma-Minda bi-univalent
functions, arXiv:1404.0895v1.

\bibitem{SSM-PTM-MOR} S. S. Miller, P. Mocanu\ and\ M. O. Reade, All $\alpha 
$-convex functions are univalent and starlike, Proc. Amer. Math. Soc. 
\textbf{37} (1973), 553--554.

\bibitem{Mocanu} P. T. Mocanu, Une propri\'et\'e de convexit\'e
g\'en\'eralis\'ee dans la th\'eorie de la repr\'esentation conforme,
Mathematica (Cluj) \textbf{11 (34)} (1969), 127--133.

\bibitem{Peng} Z. Peng, Q. Han, On the coefficients of several classes of
bi-univalent functions, Acta Mathematica Scientia, \textbf{34B} (2014),
no.~1, 228--240.

\bibitem{Pom} C. Pommerenke, \textit{Univalent functions}, Vandenhoeck \&\
Ruprecht, G\"ottingen, 1975.

\bibitem{SSS-Fekete} S. Sivaprasad Kumar\ and\ V. Kumar, Fekete-Szeg\"o
problem for a class of analytic functions, Stud. Univ. Babe\c s-Bolyai Math. 
\textbf{58} (2013), no.~2, 181--188.

\bibitem{HMS-AKM-PG} H. M. Srivastava, A. K. Mishra\ and\ P. Gochhayat,
Certain subclasses of analytic and bi-univalent functions, Appl. Math. Lett. 
\textbf{23} (2010), no.~10, 1188--1192.

\bibitem{Vein-Dale} R. Vein\ and\ P. Dale, \textit{Determinants and their
applications in mathematical physics}, Applied Mathematical Sciences, 134,
Springer, New York, 1999.

\bibitem{Zaprawa} P. Zaprawa, On the Fekete-Szeg\"o problem for classes of
bi-univalent functions, Bull. Belg. Math. Soc. Simon Stevin \textbf{21}
(2014), no.~1, 169--178.

\bibitem{Zaprawa-AAA} P. Zaprawa, Estimates of initial coefficients for
bi-univalent functions, Abstr. Appl. Anal. \textbf{2014}, Art. ID 357480,
1--6.
\end{thebibliography}
\end{document}